\documentclass[12pt,reqno]{amsart}

\newtheorem{theorem}{Theorem}[section]
\newtheorem{proposition}[theorem]{Proposition}

\newtheorem{corollary}[theorem]{Corollary}

\numberwithin{equation}{section}

\begin{document}
\baselineskip=15.5pt

\title[Hermitian symmetric space, flat bundle and holomorphicity 
criterion]{Hermitian symmetric space, flat bundle and holomorphicity criterion}

\author[H. Azad]{Hassan Azad}

\address{Department of Mathematics and Statistics, King Fahd University,
Saudi Arabia}

\email{hassanaz@kfupm.edu.sa}

\author[I. Biswas]{Indranil Biswas}

\address{School of Mathematics, Tata Institute of Fundamental
Research, Homi Bhabha Road, Bombay 400005, India}

\email{indranil@math.tifr.res.in}

\author[C. S. Rajan]{C. S. Rajan}

\address{School of Mathematics, Tata Institute of Fundamental
Research, Homi Bhabha Road, Bombay 400005, India}

\email{rajan@math.tifr.res.in}

\author[S. Sikander]{Shehryar Sikander}

\address{International Centre for Theoretical Physics, Strada Costiera 11, 34151 
Trieste, Italy}

\email{shehryar.sikander1@gmail.com}



\date{}

\begin{abstract}
Let $K\backslash G$ be an irreducible Hermitian symmetric space of noncompact type 
and $\Gamma \,\subset\, G$ a closed torsionfree discrete subgroup. Let $X$ be a compact 
K\"ahler manifold and $\rho\, :\, \pi_1(X, x_0)\,\longrightarrow\, \Gamma$ a 
homomorphism such that the adjoint action of $\rho(\pi_1(X, x_0))$ on 
$\text{Lie}(G)$ is completely reducible. A theorem of Corlette associates to $\rho$ 
a harmonic map $X\, \longrightarrow\, K\backslash G/\Gamma$. We give a criterion for 
this harmonic map to be holomorphic. We also give a criterion for it to be 
anti--holomorphic.
\end{abstract}

\maketitle

\section{Introduction}

Let $G$ be a noncompact simple Lie group of adjoint type and $K\, \subset\, G$ a 
maximal compact subgroup, such that $K\backslash G$ is an irreducible Hermitian 
symmetric space of noncompact type. The Lie algebra of $G$ will be denoted by
$\mathfrak g$. Let $\Gamma$ be a closed torsionfree discrete subgroup
of $G$. Take a compact connected K\"ahler manifold $X$;
fix a base point $x_0\, \in\, X$. Let
$$
\rho\, :\, \pi_1(X, x_0)\,\longrightarrow\, \Gamma
$$
be a homomorphism such that the adjoint action of $\rho(\pi_1(X, x_0))$ on
$\mathfrak g$ is completely reducible. This $\rho$ produces a $C^\infty$
principal $G$--bundle $E_G\, \longrightarrow\, X$ equipped with a flat connection
$D$. A reduction of structure group of $E_G$ to $K$ is given by a map
$X\,\longrightarrow\, K\backslash G/\Gamma$. Note that $K\backslash G/\Gamma$ is
a K\"ahler manifold; it need not be compact.

A theorem of Corlette says that there is a $C^\infty$ reduction of structure group 
of $E_G$ to $K$ such that corresponding map $H_D\, :\, X\,\longrightarrow\, 
K\backslash G/\Gamma$ is harmonic \cite{Co}. Our aim here is to address the 
following:
\begin{itemize}
\item When is $H_D$ holomorphic?

\item When is $H_D$ anti--holomorphic?
\end{itemize}

Let $H\, :\, X\,\longrightarrow\, K\backslash G/\Gamma$ be a $C^\infty$ map giving a
reduction of structure group of $E_G$ to $K$. We give a criterion under which $H$
is holomorphic or anti--holomorphic (see Theorem \ref{thm1}). Since a
holomorphic or anti--holomorphic map between K\"ahler manifolds is harmonic,
Theorem \ref{thm1} gives criterion for $H_D$ to be holomorphic or anti--holomorphic
(see Corollary \ref{cor3}).

\section{Harmonic map to $G/K$}\label{sec2}

A Lie group is called simple if its Lie algebra is so \cite[page 131]{He}. Let $G$ 
be a connected real noncompact simple Lie group whose center is 
trivial. It is known that if $K_1$ and $K_2$ are two maximal compact subgroups of 
$G$, then there is an element $g\, \in\, G$ such that $K_1\,=\, g^{-1}K_2g$ 
\cite[page 256, Theorem 2.2(ii)]{He}. In particular, any two maximal compact 
subgroups of $G$ are isomorphic. Assume that $G$ satisfies the condition that the 
dimension of the center of a maximal compact subgroup $K$ of it is positive. This in 
fact implies that the center of $K$ is isomorphic to $S^1\,=\, {\mathbb R}/\mathbb 
Z$ \cite[page 382, Proposition 6.2]{He}.

Fix a maximal compact subgroup $K\, \subset\, G$. Consider the left--translation 
action of $K$ on $G$. The above conditions on $G$ imply that the corresponding 
quotient $K\backslash G$ is an irreducible Hermitian symmetric space of noncompact 
type. Conversely, given any irreducible Hermitian symmetric space of noncompact 
type, there is a group $G$ of the above type such that $K\backslash G$ is 
isometrically isomorphic to it \cite[page 381, Theorem 6.1(i)]{He}. In fact, this 
gives a bijection between the isomorphism classes of Lie groups of the above type 
and the holomorphic isometry classes of irreducible Hermitian symmetric spaces 
\cite[page 381, Theorem 6.1(i)]{He}. The Lie algebra of $G$ will be denoted by 
$\mathfrak g$.

Fix a closed torsionfree discrete subgroup
$\Gamma\, \subset\, G$. Therefore, the two-sided quotient
\begin{equation}\label{eq1}
M_\Gamma\, :=\, K\backslash G/\Gamma
\end{equation}
is a connected K\"ahler manifold.

Let $X$ be a compact connected orientable real manifold. Fix a base point $x_0\, 
\in\, X$. Let
\begin{equation}\label{eq3}
\beta\, :\, \widetilde{X}\,\longrightarrow\, X
\end{equation}
be the universal cover of $X$ associated to $x_0$. The right--action of the
fundamental group $\pi_1(X, x_0)$ on $\widetilde{X}$ will be denoted by
``$\cdot$''. Let
\begin{equation}\label{eq2}
\rho\, :\, \pi_1(X, x_0)\,\longrightarrow\, \Gamma
\end{equation}
be a homomorphism such that the adjoint action of $\rho(\pi_1(X, x_0))\, \subset\, G$
on the Lie algebra $\mathfrak g$ is completely reducible. Associated to $\rho$
we have a principal $G$--bundle $E_G\,\longrightarrow\, X$ equipped with a flat
connection $D$. We note that $E_G$ is the quotient of $\widetilde{X}\times G$ (see
\eqref{eq3}) where two points $(x_1\, ,g_1)\, , (x_2\, ,g_2)\, \in\, \widetilde{X}\times G$
are identified if there is an element $z\, \in\, \pi_1(X, x_0)$ such that
$x_2\,=\, x_1\cdot z$ and $g_2\,=\, \rho(z)^{-1}g_1$. The natural projection
$\widetilde{X}\times G\,\longrightarrow\, \widetilde{X}$ produces the projection
$E_G\,\longrightarrow\, X$ from the quotient space, and the right--translation action
of $G$ on $\widetilde{X}\times G$ produces the action of $G$ on $E_G$. The trivial
connection on the trivial principal $G$--bundle $\widetilde{X}\times G\,\longrightarrow\,
\widetilde{X}$ descends to the connection $D$ on the principal $G$--bundle $E_G$.

The above construction gives the following:

\begin{corollary}\label{cor1}
The pulled back principal $G$--bundle $\beta^*E_G$ is identified
with the trivial principal $G$--bundle $\widetilde{X}\times G$ by sending any
$(x\, ,g)\, \in\, \widetilde{X}\times G$ to $(x\, ,z)$, where $z$ is the equivalence
class of $(x\, ,g)$. This identification takes the connection $\beta^*D$ on
$\beta^*E_G$ to the trivial connection on the trivial principal
$G$--bundle $\widetilde{X}\times G$.
\end{corollary}

A Hermitian structure on $E_G$ is a $C^\infty$ reduction of structure group
$$
E_K\, \subset\, E_G
$$
to the maximal compact subgroup $K\, \subset\, G$.
For such a reduction of structure group $E_K$, we have the reduction of
structure group
$$
\beta^*E_K\, \subset\, \beta^*E_G \,=\, \widetilde{X}\times G\,\longrightarrow\,
\widetilde{X}\, ,
$$
where $\beta$ is the projection in \eqref{eq3}. Therefore, a Hermitian structure on $E_G$
is given by a $\pi_1(X, x_0)$--equivariant $C^\infty$ map
\begin{equation}\label{h}
h\, :\, \widetilde{X}\, \longrightarrow G/K
\end{equation}
with $\pi_1(X, x_0)$ acting on the left of $G/K$ via the homomorphism $\rho$.

Let $\iota\, :\, G/K\, \longrightarrow\, K\backslash G$ be the isomorphism
that sends a coset $gK$ to the coset $Kg^{-1}$. Let
$$
E_K\, \subset\, E_G
$$
be the reduction of structure group to $K$ corresponding to a map $h$ as in \eqref{h}.
Since $h$ is $\pi_1(X, x_0)$--equivariant, the composition
$\iota\circ h$ descends to a map
\begin{equation}\label{h2}
H_D\, :\, X\, \longrightarrow\, M_\Gamma\,=\, K\backslash G/\Gamma
\end{equation}
(see \eqref{eq1}).

Fix a Riemannian metric $g_X$ on $X$. A theorem of Corlette says that there is a 
Hermitian structure $E_K\, \subset\, E_G$ such that the above map $h$ is harmonic 
with respect to $\beta^* g_X$ and the natural invariant metric on $G/K$ \cite[page 
368, Theorem 3.4]{Co} (in the special case where $\dim X\,=\, 2$ and $G\,=\,
\text{PSL}(2,{\mathbb R})$, this was proved in \cite{Do}).
Note that $h$ is harmonic if and only if $H_D$ in \eqref{h2} 
is harmonic with respect to $g_X$ and the natural metric on $M_\Gamma$. 
If $E'_K$ is another Hermitian structure such that corresponding map 
$\widetilde{X}\, \longrightarrow\, G/K$ is also harmonic, then there is an 
automorphism $\delta$ of the principal $G$--bundle $E_G$ such that
\begin{itemize}
\item $\delta$ preserves the flat connection $D$ on $E_G$, and

\item $\delta(E_K)\,=\, E'_K$.
\end{itemize}
In other words, if $h'\, :\, \widetilde{X}\, \longrightarrow\, G/K$ is
the $\pi_1(X, x_0)$--equivariant map corresponding to this $E'_K$, then
then there is an element $g\, \in\, G$ such that
\begin{itemize}
\item $g$ commutes with $\rho(\pi_1(X, x_0))$, and

\item $h'(y)\,=\, gh(y)$ for all $y\, \in\, \widetilde{X}$.
\end{itemize}
Note that the first condition that $g$ commutes with $\rho(\pi_1(X, x_0))$
implies that the map $y\, \longmapsto\, gh(y)$ intertwines the
actions of $\pi_1(X, x_0)$ on $\widetilde{X}$ and $G/K$.

\section{Constructions using flat connection and Hermitian structure}

Let $G$ and $K$ be as before. As before, the Lie algebra of $G$ will be denoted by
$\mathfrak g$. The Lie algebra of $K$ will be denoted by $\mathfrak k$. We have
the isotypical decomposition
\begin{equation}\label{f1}
{\mathfrak g}\,=\, {\mathfrak k} \oplus {\mathfrak p}
\end{equation}
for the adjoint action of the center of $K$ \cite[p. 208, Theorem 3.3(iii)]{He}. Note
that the adjoint action of $K$ on $\mathfrak g$ preserves this decomposition. Let
\begin{equation}\label{f2}
\phi\, :\, {\mathfrak g}\,\longrightarrow\, {\mathfrak k}\ \ \text{ and }\ \
\psi\, :\, {\mathfrak g}\,\longrightarrow\, {\mathfrak p}
\end{equation}
be the projections associated to the decomposition in \eqref{f1}.

Let $Y$ be a connected complex manifold, $F_G$ a $C^\infty$ principal 
$G$--bundle on $Y$ and $D^Y$ a flat connection on the principal $G$--bundle $F_G$. 
Fix a $C^\infty$ reduction of structure group of $F_G$
$$
F_K\, \subset\, F_G
$$
to the subgroup $K$. Let
\begin{equation}\label{f3}
\text{ad}(F_K)\,=\, F_K({\mathfrak k})\, :=\,
F_K\times^K{\mathfrak k}\, \longrightarrow\, Y\ \ \text{ and }
F_K({\mathfrak p})\, :=\,
F_K\times^K{\mathfrak p}\, \longrightarrow\, Y
\end{equation}
be the vector bundles associated to this principal $K$--bundle $F_K$ for the 
$K$--modules $\mathfrak k$ and $\mathfrak p$ respectively.
Let $\text{ad}(F_G)\,=\, F_G\times^G{\mathfrak g}\,\longrightarrow\, Y$ be the
adjoint vector bundle for $F_G$. From \eqref{f1} it follows that
\begin{equation}\label{e-11}
\text{ad}(F_G)\,=\, \text{ad}(F_K)\oplus F_K({\mathfrak p})\, .
\end{equation}

The connection $D^Y$ on $F_G$ is defined by a $\mathfrak g$--valued $1$--form on the 
total space of $F_G$; this form on $F_G$ will also be denoted by $D^Y$. The 
restriction of this form $D^Y$ to $F_K$ will be denoted by $D'$. Consider the 
$\mathfrak k$--valued $1$--form $\phi\circ D'$ on $F_K$, where $\phi$ is the 
projection in \eqref{f2}. It is $K$--equivariant and restricts to the Maurer--Cartan 
form on the fibers of the principal $K$--bundle $F_K$. Hence $\phi\circ D'$ is a 
connection on $F_K$; we will denote by $D_K$ this connection on $F_K$. Consider the 
$\mathfrak p$--valued $1$--form $\psi\circ D'$ on $F_K$, where $\psi$ is the 
projection in \eqref{f2}. It is $K$--equivariant and its restriction to any fiber of 
the principal $K$--bundle $F_K$ vanishes identically. Therefore, $\psi\circ D'$ is a 
$C^\infty$ section $F_K({\mathfrak p})\otimes T^*Y$, where $F_K({\mathfrak p})$ is 
the vector bundle defined in \eqref{f3}, and $T^*Y$ is the real cotangent bundle of 
$Y$. Let
\begin{equation}\label{h4}
D^{Y,\mathfrak p}\, \in\, C^\infty(Y,\, F_K({\mathfrak p})\otimes T^*Y)
\end{equation}
be this $F_K({\mathfrak p})$--valued $1$--form on $Y$.

Using the complex structure on $Y$ we may decompose any $F_K({\mathfrak p})$--valued
$1$--form on $Y$ into a sum of $F_K({\mathfrak p})$--valued forms of types $(1\, ,0)$
and $(0\, ,1)$. Let
\begin{equation}\label{f4}
D^{Y, \mathfrak p}\,=\, D^{1,0}+D^{0,1}
\end{equation}
be the decomposition of the above $F_K({\mathfrak p})$--valued $1$--form
$D^{Y,\mathfrak p}$ into $(1\, ,0)$ and $(0\, ,1)$ parts.

Define
$$
\omega_{D^Y}\, :=\, \frac{1}{\sqrt{-1}}(D^{1,0}- D^{0,1})\, \in\, C^\infty(Y,\, 
(\text{ad}(F_G)\otimes
T^*Y)\otimes_{\mathbb R}{\mathbb C})\, ,
$$
where $D^{1,0}$ and $D^{0,1}$ are constructed in \eqref{f4}. From \eqref{f4} it
follows that
\begin{equation}\label{f5}
\omega_{D^Y}\,\in\,C^\infty(Y,\, F_K({\mathfrak p})\otimes T^*Y)\,\subset\,
C^\infty(Y,\, (\text{ad}(F_G)\otimes T^*Y)\otimes_{\mathbb R}{\mathbb C})
\end{equation}
(see \eqref{e-11}).
In other words, $\omega_{D^Y}$ is a real $1$--form on $Y$ with values in the real vector
bundle $F_K({\mathfrak p})\,\subset\,\text{ad}(F_G)\otimes_{\mathbb R}{\mathbb C}$.

\section{Criterion for holomorphicity}\label{se4}

We now work in the set-up of Section \ref{sec2}. Assume that the manifold $X$ is 
equipped with a complex structure. Identify two elements $(g_1\, , g'_1)$ and 
$(g_2\, , g'_2)$ of $(K\backslash G)\times G$ if there is an element $\gamma\, \in\, 
\Gamma$ such that $g_2\,=\, g_1\gamma^{-1}$ and $g'_2\,=\, \gamma g'_1$. The 
corresponding quotient will be denoted by ${\mathcal E}_G$. The group $G$ acts on 
${\mathcal E}_G$; the action of any $g\, \in\, G$ sends the equivalence class of 
$(g_1\, , g'_1)\, \in\, (K\backslash G)\times G$ to the equivalence class of $(g_1\, 
, g'_1g)$. Consider the map $${\mathcal E}_G\, \longrightarrow\, M_\Gamma\,=\, 
K\backslash G/\Gamma$$ that sends the equivalence class of any $(g\, , g')\, \in\, 
(K\backslash G)\times G$ to the equivalence class of $g$. It, and the above action 
of $G$ on ${\mathcal E}_G$, together make ${\mathcal E}_G$ a principal $G$--bundle 
on $M_\Gamma$. Pull back the left invariant Maurer--Cartan form on $G$ using the 
projection to the second factor $(K\backslash G)\times G\, \longrightarrow\, G$. 
This pulled back form descends to the quotient space ${\mathcal E}_G$. It is 
straight-forward to check that this descended form defines a connection on the 
principal $G$--bundle ${\mathcal E}_G$. This connection on ${\mathcal E}_G$ will be 
denoted by ${\mathcal D}^0$. The connection ${\mathcal D}^0$ is flat.

Consider the submanifold $\{(g^{-1}\, ,g)\,\mid\, g\, \in\, G\}\,\subset\, G\times
G$. Let
\begin{equation}\label{M}
{\mathcal N}\,\subset\, (K\backslash G)\times G
\end{equation}
be the image of it under the obvious quotient map. Let
\begin{equation}\label{ek}
{\mathcal E}_K\, \subset\, {\mathcal E}_G
\end{equation}
be the image of $\mathcal N$ in
the quotient space ${\mathcal E}_G$ (recall that ${\mathcal E}_G$ is a quotient of
$(K\backslash G)\times G$).
It is straight-forward to check that the action of the subgroup $K\, \subset\, G$
on ${\mathcal E}_G$ preserves ${\mathcal E}_K$. More precisely, ${\mathcal E}_K$
is a reduction of structure group of the principal $G$--bundle ${\mathcal E}_G
\,\longrightarrow\, M_\Gamma$ to the subgroup $K\,\subset\, G$.

Take $\rho$ as in \eqref{eq2}. As in Section \ref{sec2}, let $(E_G\, ,D)$ denote 
the associated flat principal $G$--bundle over the compact complex manifold $X$. 
Take any $C^\infty$ reduction of structure group
$$
E_K\, \subset\, E_G
$$
to $K\, \subset\, G$.

\begin{proposition}\label{prop1}
The pulled back principal $G$--bundle $H^*_D {\mathcal E}_G$, where $H_D$ is 
constructed in \eqref{h2}, is canonically identified with the principal $G$--bundle
$E_G$. This identification between $E_G$ and $H^*_D {\mathcal E}_G$ takes
\begin{enumerate}
\item the pulled back connection $H^*_D {\mathcal D}^0$ to
the connection $D$ on $E_G$, and

\item the reduction $H^*_D {\mathcal E}_K\, \subset\, H^*_D{\mathcal E}_G$ (see
\eqref{ek}) to the reduction $E_K\, \subset\, E_G$.
\end{enumerate}
\end{proposition}

\begin{proof}
Consider the pulled back reduction of structure group $\beta^*E_K\, \subset\, 
\beta^*E_G$, where $\beta$ is the universal cover in \eqref{eq3}. From Corollary 
\ref{cor1} we know that $\beta^*E_G\,=\, \widetilde{X}\times G$. Let
$$
f'\,:\, \beta^*E_K\, \longrightarrow\, G
$$
be the composition
$$
\beta^*E_K\, \hookrightarrow\, \beta^*E_G \,\stackrel{\sim}{\longrightarrow}\,
\widetilde{X}\times G\,\,\stackrel{{\rm pr}_2}{\longrightarrow}\, G\, ,
$$
where ${\rm pr}_2$ is the projection to the second factor. Now define the map
$$
f''\, :\, \beta^*E_K\, \longrightarrow\, (K\backslash G)\times G\, ,\ \  z\,
\longmapsto\, (\widetilde{f'(z)^{-1}}\, , f'(z))\, ,
$$
where $\widetilde{f'(z)^{-1}}\, \in\, K\backslash G$ is the image of $f'(z)^{-1}$
under the quotient map $G\, \longrightarrow\, K\backslash G$.
The Galois group $\pi_1(X, x_0)$ for $\beta$ has a natural right--action of the pullback
$\beta^*E_K$ that lifts the right--action of $\pi_1(X, x_0)$ on $\widetilde{X}$.
For any $z\, \in\, \beta^*E_K$ and $\gamma\,\in\, \pi_1(X, x_0)$, we have
$$
f''(z\cdot\gamma)\,=\, (\widetilde{f'(z)^{-1}}\rho(\gamma)
\, , \rho(\gamma)^{-1}f'(z))\, ,
$$
where $z\cdot\gamma\,\in\, \beta^*E_K$ is the image of $z$ under the action of
$\gamma$. This, and the fact that $\text{image}(f'')\, \subset\, {\mathcal N}$ (defined
in \eqref{M}), together imply that $f''$ descends to a map
$$
\widetilde{d}\, :\, E_K\, \longrightarrow\, \mathcal{E}_K
$$
(see \eqref{ek}). This map $\widetilde{d}$ is clearly $K$--equivariant, and
the following diagram is commutative
$$
\begin{matrix}
E_K & \stackrel{\widetilde{d}}{\longrightarrow} & \mathcal{E}_K\\
\Big\downarrow && \Big\downarrow\\
X & \stackrel{H_D}{\longrightarrow}& M
\end{matrix}
$$
Therefore, we get an isomorphism of principal $K$--bundles
\begin{equation}\label{f}
f\, :\, E_K\, \longrightarrow\, H^*_D \mathcal{E}_K\, .
\end{equation}
Since $E_G$ (respectively, ${\mathcal E}_G$) is the extension of structure group
of $E_K$ (respectively, ${\mathcal E}_K$) using the inclusion of $K$ in $G$, the
isomorphism $f$ in \eqref{f} produces an isomorphism of principal $G$--bundles
$$
\widehat{f}\, :\, E_G\, \longrightarrow\, H^*_D \mathcal{E}_G\, .
$$
It is straight--forward to check that $\widehat{f}$ takes the connection $D$ to
the connection $H^*_D\mathcal{D}^0$.
\end{proof}

Define ${\mathcal E}_K({\mathfrak p})\, :=\, {\mathcal E}_K\times^K{\mathfrak 
p}\,\longrightarrow\, M_\Gamma$ as in \eqref{f3}. Also, construct the 
associated vector bundle $E_K({\mathfrak p})$ on $X$ as done in \eqref{f3}.

\begin{corollary}\label{cor2}
The pulled back vector bundle $H^*_D {\mathcal E}_K({\mathfrak p})$, where $H_D$
is defined in \eqref{h2}, is canonically isomorphic to $E_K({\mathfrak p})$.
\end{corollary}

\begin{proof}
The isomorphism $f$ in \eqref{f} between principal $K$--bundles produces an isomorphism
between the associated vector bundles.
\end{proof}

Let
\begin{equation}\label{hj}
\omega_{{\mathcal D}^0}\,\in\,C^\infty(M_\Gamma,\,
{\mathcal E}_K({\mathfrak p})\otimes_{\mathbb R}T^*M_\Gamma)
\end{equation}
be the real $1$--form constructed as in \eqref{f5} for the pair $({\mathcal D}^0\, ,
{\mathcal E}_K)$. Using the isomorphisms in Proposition \ref{prop1} and Corollary
\ref{cor2}, together with the homomorphism
$$
(dH_D)^*\, :\, H^*_DT^*M_\Gamma\, \longrightarrow\, T^*X\, ,
$$
the pulled back section $H^*_D\omega_{{\mathcal D}^0}$ produces a section
\begin{equation}\label{f6}
\widetilde{H^*_D\omega_{{\mathcal D}^0}}\,\in\,C^\infty(X,\,
E_K({\mathfrak p})\otimes_{\mathbb R}T^*X)\, .
\end{equation}
Let
\begin{equation}\label{f7}
\omega_D\,\in\,C^\infty(X,\, E_K({\mathfrak p})\otimes T^*X)
\end{equation}
be the real $1$--form constructed as in \eqref{f5} for the pair $(D\, , E_K)$.

\begin{theorem}\label{thm1}
The map $H_D$ is holomorphic if and only if $\widetilde{H^*_D\omega_{{\mathcal D}^0}}\,=\,
\omega_D$ (see \eqref{f6} and \eqref{f7}).

The map $H_D$ is anti--holomorphic if and only if $\widetilde{H^*_D\omega_{{\mathcal D}^0}}
\,=\, -\omega_D$.
\end{theorem}

\begin{proof}
The real tangent bundle $TM_\Gamma$ of $M_\Gamma\,=\, K\backslash G/\Gamma$ is 
identified with the associated vector bundle ${\mathcal E}_K({\mathfrak p})$, where 
${\mathcal E}_K$ is the principal $K$--bundle in \eqref{ek}. Also, we have
$H^*_D {\mathcal E}_K\,=\, E_K$ by Proposition \ref{prop1}.
Combining these we conclude that the vector bundle $E_K({\mathfrak p})\, 
\longrightarrow\, X$ in \eqref{f3} is identified with the pullback $H^*_D 
TM_\Gamma$. The $E_K({\mathfrak p})$--valued $1$--form
$D^{\mathfrak p}\, \in\, C^\infty(X,\, E_K({\mathfrak p})\otimes T^*X)$,
obtained by substituting $(X\, ,D\, , E_K)$ in place of $(Y\, ,D^Y\, ,F_K)$
in \eqref{h4}, coincides with the section given by the differential
$$
dH_D\, :\, TX\, \longrightarrow\, H^*_D TM_\Gamma\, .
$$
{}From this it follows that $\omega_D$ in \eqref{f7} is the section given by the
homomorphism
$$
TX\, \longrightarrow\, H^*_D TM_\Gamma\, ,\ \ v \, \longmapsto\,
dH_D (J_X(v))\, ,
$$
where $J_X\, :\, TX\, \longrightarrow\, TX$ is the almost complex structure on $X$.

Let
$$
{\mathcal D}^{0,\mathfrak p}\, \in\, C^\infty(M_\Gamma,\, {\mathcal E}_K({\mathfrak p})
\otimes T^*M_\Gamma)
$$
be the section constructed just as $D^{Y,\mathfrak p}$ is constructed in \eqref{h4}
after substituting $({\mathcal E}_G\, , {\mathcal D}^0\, , {\mathcal E}_K)$
in place of $(Y\, ,D^Y\, ,F_K)$. This ${\mathcal D}^{0,\mathfrak p}$ coincides with the
section given by the identity map of $TM_\Gamma$ (recall that $TM_\Gamma$ is identified
with ${\mathcal E}_K({\mathfrak p})$). Therefore, the section $\omega_{{\mathcal D}^0}$
in \eqref{hj} coincides with the section given by the almost complex structure
$$
J_{M_\Gamma}\, :\, TM_\Gamma\, \longrightarrow\, TM_\Gamma
$$
on $M_\Gamma$. Consequently, $\widetilde{H^*_D\omega_{{\mathcal D}^0}}\,=\,
\omega_D$ (respectively, $\widetilde{H^*_D\omega_{{\mathcal D}^0}}\,=\,
-\omega_D$) if and only if the differential $dH_D$ takes the almost complex
structure $J_X$ to $J_{M_\Gamma}$ (respectively, $-J_{M_\Gamma}$).
This completes the proof.
\end{proof}

Now assume that the complex manifold $X$ is K\"ahler. This means that $X$ is
equipped with a Hermitian structure $g_X$ such that
\begin{itemize}
\item the almost complex structure $J_X$ on $X$ is orthogonal with respect to $g_X$, and

\item the $(1\, ,1)$ on $X$ associate to the pair $(g_X\, , J_X)$ is closed.
\end{itemize}

If $A$ and $B$ are K\"ahler manifolds, then any holomorphic map $A\, \longrightarrow\,
B$ is harmonic \cite[page 1, \S~(1.2)(e)]{EL}, \cite{Li}. This implies that
any anti--holomorphic map $A\, \longrightarrow\, B$ is also harmonic; to see this
simply replace the almost complex structure $J_A$ of $A$ by the almost complex
structure $-J_A$. Therefore, Theorem \ref{thm1} has the following corollary:

\begin{corollary}\label{cor3}
Take the Hermitian structure on $E_G$ given by a map
$H_D\, :\, X\, \longrightarrow\, K\backslash G/\Gamma$. If
$\widetilde{H^*_D\omega_{{\mathcal D}^0}}\,=\, \omega_D$, then the map
$$
X\, \longrightarrow\, K\backslash G/\Gamma
$$
given by a harmonic metric in \cite{Co} is holomorphic. In that case
$H_D$ gives a harmonic metric. Conversely, if $H_D$ gives a harmonic metric
and $\widetilde{H^*_D\omega_{{\mathcal D}^0}}\,=\, \omega_D$, then
$H_D$ is holomorphic.

Take the Hermitian structure on $E_G$ given by a map 
$H_D\, :\, X\, \longrightarrow\, K\backslash G/\Gamma$.
If $\widetilde{H^*_D\omega_{{\mathcal D}^0}}\,=\, -\omega_D$, then the map
$X\, \longrightarrow\, K\backslash G/\Gamma$
given by a harmonic metric in \cite{Co} is anti--holomorphic. In that case
$H_D$ gives a harmonic metric. Conversely, if $H_D$ gives a harmonic metric
and $\widetilde{H^*_D\omega_{{\mathcal D}^0}}\,=\, -\omega_D$, then
$H_D$ is anti--holomorphic.
\end{corollary}



\begin{thebibliography}{AAAA}

\bibitem[Co]{Co} K. Corlette, Flat $G$-bundles with canonical metrics, \textit{Jour. 
Diff. Geom.} \textbf{28} (1988), 361--382.

\bibitem[Do]{Do} S. K. Donaldson, Twisted harmonic maps and the self-duality 
equations, \textit{Proc. London Math. Soc.} \textbf{55} (1987), 127--131.

\bibitem[EL]{EL} J. Eells and L. Lemaire, A report on harmonic maps, \textit{Bull.
London Math. Soc.} {\bf 10} (1978), 1--68.

\bibitem[He]{He} S. Helgason, \textit{Differential Geometry, Lie
Groups, and Symmetric Spaces}, Graduate Studies in Mathematics,
34. American Mathematical Society, Providence, RI, 2001

\bibitem[Li]{Li} A. Lichnerowicz, Applications harmoniques et vari\'et\'es
k\"ahleriennes, 1968/1969 Symposia Mathematica, Vol. III (INDAM, Rome, 1968/69),
pp. 341--402, Academic Press, London.

\end{thebibliography}
\end{document}